\documentclass[a4paper,11pt]{amsart}

\usepackage{latexsym, amssymb,amsthm,amsmath,color}
\usepackage{xspace}
\usepackage{tikz-cd}
\usepackage{mathtools}
\usepackage[all]{xy}
\usepackage[abs]{overpic}
\usepackage{hyperref}

\newcommand{\cc}{\ensuremath{\mathbb{C}}\xspace}

\newcommand{\rr}{\ensuremath{\mathbb{R}}\xspace}

\newcommand{\abs}[1]{\left| #1 \right|}

\newcommand{\set}[1]{\lbrace #1 \rbrace}
\newcommand{\comp}{\hspace{0.1 mm} \circ \hspace{0.1 mm}}
\newcommand{\rest}[2]{{#1}|_{#2}}
\renewcommand{\bar}{\overline}
\renewcommand{\tilde}{\widetilde}

\newcommand{\inp}[1]{\left\langle #1 \right\rangle}

\DeclareMathOperator{\Res}{Res}
\DeclareMathOperator{\res}{res}

\newcommand{\elli}{\mathcal{A}_{\abs{D}}}
\DeclareMathOperator{\At}{At}

\DeclareMathOperator{\Hess}{Hess}

\DeclareMathOperator{\Fr}{Fr}

\DeclareMathOperator{\lin}{lin}

\DeclarePairedDelimiter\floor{\lfloor}{\rfloor}
\renewcommand{\phi}{\varphi}
\usepackage{thmtools}
\declaretheorem[style=definition,qed=$\diamondsuit$]{definition}
\declaretheorem[style=definition,qed=$\triangle$,sibling=definition]{example}
\declaretheorem[style=plain,sibling=definition]{theorem}
\declaretheorem[style=plain,sibling=definition]{lemma}
\declaretheorem[style=plain,sibling=definition]{proposition}
\declaretheorem[style=plain,sibling=definition]{corollary}
\declaretheorem[style=definition,qed=$\diamondsuit$,sibling=example]{claim}
\declaretheorem[style=definition,qed=$\diamondsuit$,sibling=claim]{remark}
\newtheorem*{claim*}{Claim}
\numberwithin{theoremalpha}{section}

\numberwithin{equation}{section}
\numberwithin{definition}{section}
\numberwithin{theorem}{section}
\numberwithin{proposition}{section}
\numberwithin{lemma}{section}
\numberwithin{example}{section}
\numberwithin{remark}{section}
\numberwithin{corollary}{section}
\title{The cohomology of the elliptic tangent bundle}
\author{Aldo Witte}
\address{Department of Mathematics, Utrecht University, 3508 TA Utrecht, The Netherlands}
\email{g.a.witte@uu.nl}

\newtheoremstyle{named}{}{}{\itshape}{}{\bfseries}{.}{.5em}{\thmnote{#3}#1}
\theoremstyle{named}

\begin{document}
\begin{abstract}
In this note we compute the cohomology of the elliptic tangent bundle, a Lie algebroid introduced in \cite{CG17,CKW20} used to describe singular symplectic forms arising from generalized complex geometry.
\end{abstract}
\maketitle
\tableofcontents

\section*{Introduction}
Generalized complex structures \cite{H03,Gua07} are a simultaneous generalisation of symplectic and complex structures. Infinitesimally these structures decompose each tangent space as the direct sum of a complex and a symplectic vector space. However, the number of complex directions, called the \emph{type}, might vary throughout the manifold. A well-behaved and interesting class of structures for which the type changes, called \emph{self-crossing stable generalized complex structures}, was defined in \cite{CG17,CKW20}. These are generalized complex structures which have type zero on an open and dense set and allow the type to change along an immersed codimension-two submanifold.

In \cite{CKW20} it is shown that these self-crossing stable generalized complex structures are in one-to-one correspondence with (a subset of) \emph{self-crossing elliptic symplectic structures}. These are symplectic structures with well-behaved singularities, and can be described as Lie algebroid symplectic forms for a certain Lie algebroid: the \emph{self-crossing elliptic tangent bundle}.

Much research on Lie algebroid symplectic structures has been carried out over the past few years. Examples include b-symplectic \cite{GMP14}, scattering-symplectic \cite{Lan21}, c-symplectic \cite{MS20} (or self-crossing b-symplectic). In the algebraic/complex setting holomorphic log symplectic structures, which are holomorphic forms with logarithmic singularities in the sense of \cite{Del71}, are well-studied.

In \cite{CG17} deformations of stable generalized complex structures with embedded type change locus and the corresponding elliptic symplectic structures are described. It is shown there that deformations are completely controlled by the cohomology of the relevant Lie algebroid.\footnote{The description of deformations of stable generalized complex structures was obtained independently in \cite{G13}.}

 If we allow the type change locus to be immersed this will no longer be the case, still the computation of the Lie algebroid cohomology is the first step in understanding the deformation theory.

In this note we compute the cohomology of the self-crossing elliptic tangent bundles. This cohomology is determined by the topological data of the manifold, of the type change locus and of the normal bundle.\\
%If $M$ is a smooth algebraic variety, and $D$ a (local) normal crossing divisor in $M$ then one can consider the meromorphic forms with simple poles along $D$, denoted $\Omega^{\bullet}(\log D)$. The cohomology of these forms is well-known to be $H^{\bullet}(M\backslash D)$, and its computation can be traced back at least to \cite{Del71}. If one takes the real or imaginary part of a form in $\Omega^{\bullet}(\log D)$, one obtains a self-crossing elliptic form. The cohomology of the elliptic tangent bundle is more involved than that of $\Omega^{\bullet}(\log D)$, but importantly still topological in nature.\\
\textbf{Organisation of the paper:} This paper is organized as follows: In Section \ref{sec:1} we recall the notion of self-crossing elliptic divisors and their associated Lie algebroids. In Section \ref{sec:2} we describe the geometric structure a self-crossing elliptic divisor induces on the directions normal to its degeneracy locus. In Section \ref{sec:residues} we will use this geometric structure to define residue maps, which we will use in Section \ref{sec:4} to compute the Lie algebroid cohomology of the elliptic tangent bundle. In Section \ref{sec:5} we compute this cohomology in some explicit examples. We end this section by giving an outlook on the deformation theory of self-crossing stable generalized complex structures.

{\bf Acknowledgements.}
The author would like to thank  his supervisor Gil Cavalcanti and Ralph Klaasse for useful discussions. The author was supported by the NWO through the Utrecht Geometry Centre Graduate Programme.

\section{Elliptic divisors}\label{sec:1}
The singularities we encounter are governed by the notion of a real divisor, a notion which is inspired by the notion of divisor in algebraic geometry. In this section we will recall the definition of elliptic divisors, and their associated Lie algebroids, the elliptic tangent bundles. We will keep our discussion brief and refer to \cite{CG17,CKW20} for more information.

\begin{definition}
A \textbf{real divisor} $(R,q)$ consists of a real line bundle $R$ together with a section $q$ with nowhere dense zero-set. Its \textbf{associated ideal} $I_{q}$ is the image of the map $q : \Gamma(R^*) \rightarrow \rr$.
\end{definition}
An \textbf{isomorphism} of divisors $(L_i,q_i)$ is a vector bundle isomorphism $\Phi : R_1 \rightarrow R_2$ covering the identity and intertwining the sections, $\Phi^*(q_2) = q_1$. A real divisor is, up to isomorphism, determined by its associated ideal and therefore we may use ideals and divisors interchangeably.
\begin{definition}[\cite{CG17}]\label{def:selliptic}
A \textbf{smooth elliptic divisor} is a real divisor $(R,q)$ such that the critical set of $q$ is a codimension two submanifold along which the normal Hessian is positive definite. We call $D = q^{-1}(0)$ the \textbf{vanishing locus}.
\end{definition}
The normal Hessian is the leading part of the Taylor expansion of $q$ around $D$, and defines a section $\Hess^{\nu}q \in \Gamma(\text{Sym}^2N^*D \otimes R)$. Phrased in other words, Definition \ref{def:selliptic} states that $q$ locally defines a particular type of Morse-Bott function.
\begin{definition}[\cite{CKW20}]\label{def:xelliptic}
A \textbf{ self-crossing elliptic divisor} is a real divisor $\abs{D} = (R,q)$ on $M$ such that for every point $p\in M$ there exists a neighbourhood $U$ of $p$ such that
\begin{equation*}
	I_{q}(U) = I_1\cdot \ldots \cdot I_j,
\end{equation*}
where the $I_1,\ldots,I_j$ are smooth elliptic divisors with transversely intersecting vanishing loci.
\end{definition}
We call $I_1,\ldots,I_j$ as above a choice of \textbf{local normal crossing} elliptic divisors near a point $p$. If $I_q$ is globally of the form $I_1\cdot \ldots \cdot I_j$, then say that $I_q$ is a \textbf{global normal crossing} elliptic divisor.

In what follows if we write ``elliptic divisor'' it is understood to have self-crossings, if it is a smooth elliptic divisor we will explicitly state this.

The vanishing locus of an elliptic divisor is an immersed submanifold, stratified by the amount of submanifolds intersecting:
\begin{definition}
Let $I_{\abs{D}}$ be an elliptic divisor on $M$. The \textbf{intersection number} of a point $p \in M$ is the minimum of the integers $j$ from Definition~\ref{def:xelliptic} over all neighbourhoods $U$ of $p$. The \textbf{intersection number of the divisor} is the maximum of the intersection numbers of all points $p \in M$. If $I_{\abs{D}}$ has intersection number equal to $n$, the sets $D(j)$ of points of intersection number at least $j$ induce a filtration of $M$:
\begin{equation*}
	M = D(0) \supset D(1) = D \supset \cdots \supset D(n) \supset D(n+1) = \emptyset,
\end{equation*}
with induced \textbf{stratification} $D[j]:= D(j)/D(j+1)$. These strata are embedded submanifolds and consist of the points with intersection number precisely $j$.
\end{definition}
The fact that the $D(i)$ indeed induce a stratification follows readily from the normal forms of the elliptic divisor in Lemma \ref{lem:ellinormform} below. 

Restricting the divisor to the set of points with at most a given intersection number produces another divisor:
\begin{lemma}\label{lem:restrict}
If $I_{\abs{D}}$ is an elliptic divisor with intersection number $n$ and $i \leq n$, then the restriction $\rest{I_{\abs{D}}}{M\backslash D(i+1)}$ defines a divisor with intersection number $i$.
\end{lemma}
Note that on a manifold $M^m$ the maximum intersection number of an elliptic divisor is $\floor{m/2}$. Using the Morse--Bott lemma inductively we can establish the following local normal form for elliptic divisors:
\begin{lemma}[\cite{CKW20}]\label{lem:ellinormform}
Let $\abs{D} = (R,q)$ be an elliptic divisor with intersection number $n \leq \floor{m/2}$ on $M^{m}$, and let $x \in D[i]$ with $i \leq n$. Then there exists coordinates $(x_1,y_1,\ldots, x_i,y_i,w_{2i+1},\ldots,w_{2i+l})$ around $x$ such that $(R,q)$ is isomorphic to the elliptic divisor defined by the ideal $I_{\abs{D}}=\inp{(x_1^2+y_1^2)\cdot \ldots \cdot (x_i^2+y_i^2)}$.
\end{lemma}
We call $I_{\abs{D}}$ as above, the \textbf{standard elliptic divisor} with intersection number $i$ on $\rr^{2i} \times \rr^l$.

If $(R,q)$ is an elliptic divisor, then $R$ is trivialisable away from the codimension-two submanifold $D[1]$ and thus globally trivialisable:
\begin{lemma}
Let $(R,q)$ be an elliptic divisor on $M$. Then $R$ is trivialiasible, and consequently $I_q = \inp{f}$ for some $f \in C^{\infty}(M)$.
\end{lemma}
Although $R$ is trivialisable, it is not canonically so and thus we prefer to work with the ideal $I_q$ rather then one particular function generating it.

\begin{remark}[Metrics]\label{rem:metric}
Let $(R,q)$ be a smooth elliptic divisor with vanishing locus $D$, and let $f$ be any global function generating $I_{\abs{D}}$. Because $f$ is nowhere vanishing on $M\backslash D$, and $D$ has codimenion-two in $M$ the sign of $f$ is constant on the entirety of $M$. Consequently, we can always choose a non-negative representative of the elliptic ideal. Therefore the normal Hessian of $f$, $\Hess^{\nu}f \in \Gamma(\text{Sym}^2N^*D)$ is positive definite, and thus defines a metric.
\end{remark}
%
%
%
%If we let $Q_f$ denote the quadratic form associated to $\Hess^{\nu}f$ we can use the Morse--Bott lemma to obtain the following normal form for smooth elliptic divisors:
%\begin{proposition}\label{prop:ellinormform}
%Let $(R,q)$ be an elliptic divisor, and choose a trivialisation of $R$ under which $q$ corresponds to a non-negative function $h$. Then there exists a neighbourhood of the zero section of $ND$ and an open embedding $\varphi : U \rightarrow ND$ such that $\rest{\Phi}{D} = \text{Id}_D$ and $\Phi^*h = Q_h$.
%\end{proposition}
%Here, $Q_h$ is the quadratic function corresponding to the non-negative definite form $\Hess^{\nu} h$.

\begin{definition}
Let $\abs{D} = (R,q)$ be an elliptic divisor. The vector fields preserving $I_{\abs{D}}$ define a Lie algebroid $\mathcal{A}_{\abs{D}} \to TM$, called the \textbf{elliptic tangent bundle}.
\end{definition}
The existence of this Lie algebroid is established through the Serre--Swan theorem: if $x \in D[i]$, then in the coordinates of Lemma \ref{lem:ellinormform} vector fields preserving $I_{\abs{D}}$ are given by
\begin{equation*}
\Gamma(\mathcal{A}_{\abs{D}}) = \inp{r_1\partial_{r_1},\partial_{\theta_1}, \ldots,r_i\partial_{r_i},\partial_{\theta_i},\partial_{w_{2i+1}},\ldots,\partial_{w_{2i+l}}},
\end{equation*}
where $r_j\partial_{r_j} := x_j\partial_{x_j}+y_j\partial_{y_j}$ and $\partial_{\theta_j} := x_j\partial_{y_j}-y_j\partial_{x_j}$. Consequently, vector fields preserving $I_{\abs{D}}$ define a locally free sheaf and thus induce a Lie algebroid.

Lie algebroid forms for the elliptic tangent bundle are locally given by
\begin{equation*}
\Gamma(\elli^*) = \inp{ d\log r_1,d\theta_1,\ldots, d\log r_i, d\theta_i,dw_{2i+1},\ldots,dw_{2i+l}},
\end{equation*}
with $d\log r_j = (x_j^2+y_j^2)^{-1}(x_jdx_j+y_jdy_j)$ and $d\theta_j = (x_j^2+y_j^2)^{-1}(x_jdy_j-y_jdx_j)$.

\begin{definition}
Let $\abs{D} = (R,q)$ be an elliptic divisor. A \textbf{(self-crossing) elliptic symplectic structure} is a Lie algebroid symplectic form for the elliptic tangent bundle: That is, a two-form $\omega \in \Omega^2(\elli)$ which is non-degenerate ($\omega^\flat : (\elli)_x \rightarrow (\elli)_x^*$ is an isomorphism) and closed ($d_{\elli}\omega = 0$).
\end{definition}
In \cite{CKW20} we study these structures and show that a subset corresponds to a class of generalized complex structures, called self-crossing stable.

\section{Geometric structure on the normal bundle}\label{sec:2}
In this section we study the geometric structure present on the normal bundles of the strata of the vanishing locus of an elliptic divisor. In the next section, this geometric structure will be used to describe the Lie algebroid cohomology of the elliptic tangent bundle.\\

Remark \ref{rem:metric} explains how a smooth elliptic divisor $(R,q)$ induces a metric on the normal bundle to $D$. The metric depends on the choice of particular trivialisation of $R$, but the conformal class does not. Phrased in another way, $ND$ inherits a canonical $O(2)\times\rr_+$ structure group reduction. For self-crossing elliptic divisors we have the following result:

\begin{proposition}\label{prop:canstrucred}
Let $I_{\abs{D}}$ be an elliptic divisor with intersection number $n$. Then $ND[n]$ admits a canonical structure group reduction to $(O(2)\times \rr_+)^n\rtimes S_n$.
\end{proposition}
\begin{proof}
For a point $x\in D[n]$, let $I_{\abs{D_1}},\ldots,I_{\abs{D_n}}$ be any choice of local normal crossing divisors and let the $f_i$ be any representatives of the $I_{\abs{D_i}}$, which we may choose to be all non-negative. Because the $D_i$ all intersect transversely, we have 
\begin{equation*}
	N_x(D[n]) \simeq \bigoplus_{i=1}^n \rest{N_x(D_i)}{D[n]}.
\end{equation*}
If $u_i = (u_{i_1},u_{i_2})$ forms a local frame for $\rest{ND_i}{D[n]}$, then $(u_{1_1},u_{1_2},\ldots,u_{n_1},u_{n_2})$ forms a local frame for $N(D[n])$. We consider all such local frames for which $u_i$ is orthonormal with respect to $\Hess^{\nu}(f_i)$. There are two important points to remark. First, the choice of representatives $f_i$ is not canonical. Secondly, the ordering of the local divisors is not well-defined, and instead there is an $S_n$-symmetry present which permutes the divisors. Therefore we consider the frames with respect to all these choices of representatives $f_i$ and all orderings for the local divisors. This provides a reduction of the structure group to $(O(2)\times \rr_+)^n\rtimes S_n$.
\end{proof}
\begin{corollary}\label{prop:strucredfirst}
Let $I_{\abs{D}}$ be an elliptic divisor with intersection number $n$, then 
\begin{itemize}
\item the structure group of $ND[n]$ reduces to $O(2)^n \ltimes S_n$.
\item If $I_{\abs{D}}$ is a global normal crossing divisor, then a choice of non-negative representatives $f_1,\ldots,f_n$ of the smooth elliptic divisor ideals induces a further structure group reduction to $O(2)^n$.
\item If $D[1]$ is furthermore co-orientable, then a further choice of co-orientation of $D[1]$ induces a structure group reduction to $T^n$.
\end{itemize}

\end{corollary}
\begin{proof}~
\begin{itemize}
\item The quotient of $(O(2)\times \rr_+)^n\rtimes S_n$ by $O(2)^n \ltimes S_n$ is $(\rr_+)^n$, and thus contractible. Therefore the required structure group reduction exists.
\item Let $I_{\abs{D_1}},\ldots,I_{\abs{D_n}}$ be global smooth elliptic divisors for $I_{\abs{D}}$. We consider frames $(u_1,\ldots,u_n)$ of $ND[n]$, where $u_i$ is a local frame of $\rest{ND_i}{D[n]}$ orthonormal with respect to $\Hess^{\nu}(f_i)$. This will give the required $O(2)^n$-reduction.
\item The co-orientation on $D[1]$ induces a co-orientation on each $D_i$ away from a codimension-two submanifold, and therefore a co-orientation on the entirety of $D_i$. Proceeding as in the previous point with oriented frames provides a structure group reduction to $SO(2)^n$.
\end{itemize}
\end{proof}
\begin{remark}
By Lemma \ref{lem:restrict} the restriction of an elliptic divisor $I_{\abs{D}}$ to $M\backslash D(i+1)$ is an elliptic divisor of intersection number $i$. The highest stratum of this divisor is precisely $D[i]$, and consequently we can apply the results in this section to obtain structure group reduction of $ND[i]$.
\end{remark}

\section{Residue maps}\label{sec:residues}
In this section we describe residue maps for elliptic divisors. These are maps which pick out the coefficient of a singular generator of a Lie algebroid form, much akin to the residue of a meromorphic differential form. In \cite{K18} a general theory of residue maps of Lie algebroids is described. We will first recall the general definition and then specialise to the case of elliptic divisors.\\

Let $0 \rightarrow A \rightarrow B \rightarrow C \rightarrow 0$ be a short exact sequence of vector spaces with $l = \dim A$. A splitting of this sequence induces an isomorphism $B \simeq A \oplus C$. Consequently $\wedge^{k}B^*$ decomposes as a direct sum $\oplus_i (\wedge^{k-i}C^*\otimes \wedge^kA^*)$. The projections from $\wedge^kW^*$ to each of these factors depend on the particular splitting, however, the projection to $\wedge^{k-l}C^* \otimes \wedge^l A^*$ does not. We define the \textbf{residue map} to be this projection:
\begin{equation*}
\Res : \wedge^kB^* \rightarrow \wedge^{k-l}C^* \otimes \wedge^l A^*.
\end{equation*}

Let $0 \rightarrow \mathcal{A} \rightarrow \mathcal{B} \rightarrow \mathcal{C} \rightarrow 0$ be a short exact sequence of Lie algebroids with $\text{rk}(\mathcal{A}) = l$. Applying the residue map fibre-wise gives a map $\Res:\Omega^k(\mathcal{B}) \rightarrow \Omega^{k-l}(\mathcal{C};\det(\mathcal{A}^*))$. To endow $\Omega^{\bullet}(\mathcal{C};\det(\mathcal{A}^*))$ with a differential one needs a flat $\mathcal{C}$-connection on $\det(\mathcal{A}^*)$. We now describe a situation in which such a connection exists and the residue map is a cochain morphism with respect to this differential.
\begin{proposition}\label{prop:res maps}
Let $\mathcal{B}\rightarrow M$ be a Lie algebroid and let $i: \mathcal{A} \hookrightarrow \mathcal{B}$ be an abelian ideal subalgebroid of rank $l$. Then there is a canonical flat $(\mathcal{B}/\mathcal{A})$-connection on $\det(\mathcal{A}^*)$. 

Furthermore assume that for every point $x \in M$, there exists a neighbourhood $U$ of $p$ and closed sections $\beta_1,\ldots,\beta_l \in \Omega^1(\mathcal{B}|_U)$ such that $i^*\beta_1,\ldots,i^*\beta_l$ generate $\Omega^1(\mathcal{A}|_U)$. Then $\Res:\Omega^k(\mathcal{B}) \rightarrow \Omega^{k-l}(\mathcal{B}/\mathcal{A};\det(\mathcal{A}^*))$ is a cochain morphism.
\end{proposition}
\begin{proof}
Let $\sigma : \mathcal{B}/\mathcal{A} \rightarrow \mathcal{B}$ be any splitting and define:
\begin{equation*}
\nabla_{c}(a) = [\sigma(c),a], \quad c \in \Gamma(\mathcal{B}/\mathcal{A}), a \in \Gamma(\mathcal{A}). 
\end{equation*}
It is straightforward to verify that this is connection is flat and does not depend on the choice of splitting. If we let $\beta_1,\ldots,\beta_l \in \Omega^1(\mathcal{B})$ be local closed forms with the property that their restriction to $\mathcal{A}$ defines a local frame for $\Omega^1(\mathcal{A})$, then 
\begin{equation*}
\Res(\beta_1 \wedge \ldots \wedge \beta_l \wedge \gamma) = \gamma \otimes (\beta_1 \wedge \ldots \wedge \beta_l),
\end{equation*}
and one easily verifies that $\Res$ is a cochain morphism.
\end{proof}
If $I_{\abs{D}}$ is an elliptic ideal with intersection number $n$ then the restriction of the anchor of the elliptic tangent bundle to $D[n]$ has image inside $TD[n]$, and hence $\rest{\elli}{D[n]}$ is again a Lie algebroid. We will describe this Lie algebroid using the language of Atiyah algebroids, which we recall in Appendix \ref{sec:Atiyah}. 
\begin{lemma}\label{lem:restatiyah}
Let $I_{\abs{D}}$ be an elliptic divisor with intersection number $n$ and let $P$ denote the $(O(2)\times \rr_+)^n\rtimes S_n$-structure group reduction of $ND[n]$ from Proposition \ref{prop:canstrucred}. Then $\rest{\elli}{D[n]} \simeq \At(P)$.
\end{lemma}
\begin{proof}
Let $I_1,\ldots,I_n$ be local normal crossing elliptic divisors and let $f_1,\ldots,f_n$ by representatives of these ideals. Let $Q_{f_i} \in C^{\infty}(ND_i)$ denote the quadratic approximations of these functions, and define the ideal $I_{\abs{D}}^{ND[n]} = \inp{Q_{f_1}\cdots Q_{f_n}}$. Although the representatives $f_i$ are not unique nor global this ideal is and thus defines an elliptic divisor on $ND[n]$. Let $\elli^{ND[n]}$ denote the corresponding elliptic tangent bundle. And let $\Gamma(\elli^{ND[n]})_{\text{lin}}$ denote the sections which are send to linear vector fields by the anchor. Remark that $\Gamma(\elli^{ND[n]})_{\text{lin}}$ is only a $C^{\infty}(D[n])$-module, not a $C^{\infty}(ND[n])$-module.

Given $X \in \Gamma(\rest{\elli}{D[n]})$ one can show that there exists a unique linear vector field $\tilde{X} \in \Gamma(\elli^{ND[n]})_{\text{lin}}$ with $\rest{\tilde{X}}{D[n]}=X$. Existence of such a vector field is obtained by linearising any extension $X'\in \Gamma(\rest{\elli}{\mathcal{U}})$ of $X$ and observing that because $X'$ preserves the ideal $I_{\abs{D}}$ its linearisation $\tilde{X}$ preserves the ideal $I_{\abs{D}}^{ND[n]}$.
%Let $X \in \Gamma(\rest{\elli}{D[n]})$ and choose a tubular neighbourhood $\mathcal{U}$ of $D[n]$ and an extension $X'\in \Gamma(\rest{\elli}{\mathcal{U}})$ to $\mathcal{U}$. Linearise $X'$ to obtain a vector field $\tilde{X}$ on $ND[n]$, which by the nature of the elliptic tangent bundle does not depend on any of these choices. Moreover, because $X'$ preserves the ideal $I_{\abs{D}}$ it's linearisation $\tilde{X}$ preserves the ideal $I_{\abs{D}}^{ND[n]}$.
We thus obtain a map
\begin{equation*}
\phi\colon \Gamma(\rest{\elli}{D[n]}) \rightarrow \Gamma(\elli^{ND[n]})_{\text{lin}}, \qquad X \mapsto \tilde{X}.
\end{equation*}
This map is a bijection and bracket preserving. To finish the argument we are left to show that sections of the right-hand side coincide with sections of $\At(P)$. Let $F$ be a representative of $I_{\abs{D}}^{ND[n]}$, then
\begin{equation*}
\Gamma(\elli^{ND[n]})_{\text{lin}} = \set{X \in \mathfrak{X}^1(ND[n])_{\text{lin}} : \mathcal{L}_X(F) = \lambda F \text{ for some } \lambda \in C^{\infty}(D[n])}.
\end{equation*}
Note that a priori for an element of $\Gamma(\elli^{ND[n]})$ we only have $\mathcal{L}_X(F) = \lambda F$ with $\lambda \in C^{\infty}(ND[n])$, however because $X$ is linear we must have $\lambda \in C^{\infty}(D[n])$.

As $F$ is locally of the form $Q_{f_1}\cdots Q_{f_n}$ and all these functions are functionally indivisible, we must have that for all $i$ there exists precisely one $j$ such that $\mathcal{L}_X(Q_{f_i}) = \lambda^i_j Q_{f_j}$ for some function $\lambda^i_j \in C^{\infty}(D[n])$. 

On the other side, using Lemma \ref{lem:confmetautos} one can show that
\begin{equation*}
\At(P) = \set{X \in \mathfrak{X}(ND[n])_{\text{lin}} : \forall ~ i ~ \exists ! ~ j ~ \text{s.t. } \mathcal{L}_X(\Hess^{\nu}(f_i)) = \lambda^i_j \Hess^{\nu}(f_j), \text{ for some } \lambda^i_j \in C^{\infty}(D[n])}.
\end{equation*}
Using the fact that for a linear vector field $X$ and definite Morse--Bott functions $f,g$ we have $\mathcal{L}_X(\Hess^{\nu}f) = g$ if and only if $\mathcal{L}_X(Q_f) = Q_g$ we see that indeed $\Gamma(\elli^{ND[n]})_{\text{lin}}= \Gamma(\At(P))$, which finishes the proof.
\end{proof}

\subsection{Residues for smooth elliptic divisors}
In this section we describe the radial residue map for a smooth elliptic divisor. The radial residue map was already used in \cite{CG17} to compute the Lie algebroid cohomology of the elliptic tangent bundle of a smooth elliptic divisor. We recall their construction and add some details which are especially important in the general case.
\begin{lemma}\label{lem:isotropygood}
Let $I_{\abs{D}}$ be a smooth elliptic divisor on a manifold $M$, and let $\rho : \elli \rightarrow TM$ be the corresponding elliptic tangent bundle. There exists a unique nowhere vanishing section $\mathcal{E}_D$ in $\rest{\ker \rho}{D}$ with the property that every extension is an Euler like vector field\footnote{That is, a vector field which vanishes along $D$ and has linearisation the Euler vector field.}.
\end{lemma}
\begin{proof}
Let $\mathcal{U}$ be any tubular neighbourhood of $D$, and let $\mathcal{E}$ be the corresponding Euler vector field. Let $f \in I_{\abs{D}}$ be a local representative of the ideal, and pick local Morse--Bott coordinates in which $f = r^2$. In these coordinates $\mathcal{E}$ is given by $r\partial_r$. Therefore, $\mathcal{E}$ defines a section of the elliptic tangent bundle and its restriction to $D$ defines a section $\mathcal{E}_D$ of $\rest{\ker \rho}{D}$. It is readily verified that $\mathcal{E}_D$ is the unique vector field with the given property.
\end{proof}

Because $L_r = \rr \cdot \mathcal{E}_D$ is an ideal of $\rest{\elli}{D}$ we can consider the quotient Lie algebroid, which can be described as follows:

\begin{lemma}
Let $I_{\abs{D}}$ be a smooth elliptic divisor with non-negative generator $f \in I_{\abs{D}}$. Let $Q = O(2)ND$ denote the induced structure group reduction from Corollary \ref{prop:strucredfirst}. Then
\begin{align*}
\rest{\elli}{D}/L_r &\rightarrow \At(Q)\\
[X] &\mapsto X - \frac{\mathcal{L}_X(f)}{f}\mathcal{E}_D
\end{align*}
is an isomorphism of Lie algebroids.
\end{lemma}

Because $L_r$ is an abelian ideal of $\rest{\elli}{D}$ we may apply Proposition \ref{prop:res maps} to obtain a residue map $\Omega^\bullet(\rest{\elli}{D}) \rightarrow \Omega^{\bullet-1}(\At(O(2)ND);L_r^*)$. By Lemma \ref{lem:ellinormform} $\Omega^1(\elli)$ is locally generated by closed forms and therefore Proposition \ref{prop:res maps} implies that the residue map is a cochain morphism. After pre-composing with the restriction to $D$ to obtain the the \textbf{radial residue map}:
\begin{equation*}
\Res_r : \Omega^\bullet(\mathcal{A}_{\abs{D}}) \rightarrow \Omega^{\bullet-1}(\At(O(2)ND) ; L_r^*) = \Omega^{\bullet-1}(\At(O(2)ND)),
\end{equation*}
where we used that $L_r^*$ is trivial as an $\At(O(2)ND)$-representation.
Note that this residue map depends on the choice of representative $f \in I_{\abs{D}}$, but any two different choices result into isomorphic Atiyah algebroids. In the local coordinates $(x_1,y_1,w_3,\ldots,w_n)$ of Lemma \ref{lem:ellinormform} the radial residue map is given by:
\begin{equation*}
\Res_r(\alpha) = (\iota_{r_1 \partial_{r_1}} \alpha)|_{D}, \qquad \alpha \in \Omega^\bullet(\mathcal{A}_{|D|}).
\end{equation*}
\subsection{Residues for self-crossing elliptic divisors}
We will now describe residue maps for the self-crossing elliptic tangent bundle. 
As before, the Euler vector fields generate an abelian ideal:

\begin{lemma}\label{lem:atiso2}
Let $I_{\abs{D}}$ be an elliptic divisor with intersection number $n$. Then there exists a canonical abelian ideal subalgebroid $L_r$ of $\rest{\elli}{D[n]}$. In local coordinates of Lemma \ref{lem:ellinormform} we have
\begin{equation*}
\Gamma(L_r) \simeq \inp{\rest{r_1\partial_{r_1}}{D[n]},\ldots, \rest{r_n\partial_{r_n}}{D[n]}}.
\end{equation*}
\end{lemma}

\begin{proof}
Let $x \in D[n]$, and let $I_{\abs{D_1}},\ldots,I_{\abs{D_n}}$ be local normal crossing divisor on a neighbourhood $U$ of $x$ and let $\rho_i : \mathcal{A}_{\abs{D_i}} \rightarrow TU$ be the corresponding elliptic tangent bundles. By transverality of the loci $D_i$ we obtain the following splitting:
\begin{equation*}
\rest{\ker \rho}{D[n]\cap U} \simeq \rest{\ker \rho_1}{D[n]\cap U} \oplus \cdots \oplus \rest{\ker \rho_n}{D[n]\cap U}.
\end{equation*}
By Lemma \ref{lem:isotropygood} there exist canonical vector subbundles $L_{r_i}$ of $\rest{\ker \rho_i}{D(n)}$ with generators $\mathcal{E}_{D_i}$. Define $\rest{L_r}{U}$ as the direct sum of these vector subbundles. Although each of the summands $L_{r_i}$ is not globally defined and neither is their order, $L_r$ does define a global vector bundle. Finally, in the local coordinates of Lemma \ref{lem:ellinormform} we have $\mathcal{E}_{D_i} = r_i \partial_{r_i}$ as desired.
\end{proof}
Because $L_r$ is an ideal we can consider the quotient Lie algebroid which can be described as follows:
\begin{lemma}
There is an $O(2)^n \rtimes S_n$-principal bundle $Q$ such that $\At(Q) = \rest{\elli}{D[n]}/L_r$. 
\end{lemma}
%For smooth elliptic divisors there was a clear description of $Q$ and this isomorphism, unfortunately for the general case no such description exists.
\begin{proof}
Let $P = (O(2)\times \rr_+)^n \rtimes S_n$ be structure group reduction from Proposition \ref{prop:canstrucred}. As locally the sections of $L_r$ are generated by the Euler vector fields of the normal bundles to the local loci $D_i$ we have that $L_r$ generates an $(\rr_+)^n$-action on $P$. If $x \in M$ and $(u_1,\ldots,u_n) \in P_x$ is a frame with $u_i$ a frame of $N_xD_i$, then this action is given by $(\lambda_1,\ldots,\lambda_n)\cdot (u_1,\ldots,u_n) = (\lambda_1u_1,\ldots,\lambda_nu_n)$. We can consider the quotient bundle $Q := P/(\rr_+)^n$, which is an $O(2)^n \rtimes S_n$-principal bundle. Consequently $\rest{\elli}{D[n]}/L_r = \At(Q)$ which finishes the proof.
\end{proof}

As in the smooth case we apply Proposition \ref{prop:res maps} to obtain a residue map, which we compose with the restriction to $D[n]$ to obtain the \textbf{total radial residue map}:
\begin{equation*}
\Res_r^n : \Omega^\bullet(\elli) \rightarrow \Omega^{\bullet-n}(\At(Q);\det(L_{r}^*)),
\end{equation*}
In the local coordinates of Lemma \ref{lem:ellinormform} we have that the total radial residue map is given by
\begin{equation*}
\Res_r^n(\alpha) = (\iota_{r_1\partial_{r_1}\wedge \cdots \wedge r_n\partial_{r_n}}\alpha)|_{D[n]}.
\end{equation*}

When the elliptic divisor is a global normal crossing divisor the situation simplifies significantly:
\begin{lemma}\label{lem:gncresmap}
Let $I_{\abs{D}} = I_{\abs{D_1}}\cdot \ldots \cdot I_{\abs{D_n}}$ be a global normal crossing elliptic divisor. Let $Q = O(2)^nND[n]$ denote the structure group reduction from Corollary \ref{prop:strucredfirst}. Then $\rest{\elli}{D[n]}/L_r \simeq \At(Q)$. 
\end{lemma}
\begin{proof}
Let $f_1,\ldots,f_n$ be non-negative representatives of the smooth elliptic ideals. Let $L_r = L_{r_1}\oplus \cdots \oplus L_{r_n}$ be the ideal subalgebroid from Lemma \ref{lem:atiso2} and let $\mathcal{E}_{D_1},\ldots,\mathcal{E}_{D_n}$ denote the preferred generators. Let $f_1,\ldots,f_n$ be non-negative representatives of $I_{\abs{D_1}},\ldots,I_{\abs{D_n}}$. Then
\begin{align*}
\rest{\elli}{D[n]}/L_r &\rightarrow \At(Q),\\
[X] &\mapsto X - \frac{\mathcal{L}_X(f_1)}{f_1}\mathcal{E}_{D_1}-\ldots - \frac{\mathcal{L}_X(f_n)}{f_n}\mathcal{E}_{D_n},
\end{align*}
is the desired isomorphism.
\end{proof}
In this case $\det(L_{r}^*)$ is trivial as an $\At(O(2)^n)$-representation and we conclude that the total radial residue map becomes a map:
\begin{equation*}
\Res_r^n : \Omega^\bullet(\elli) \rightarrow \Omega^{\bullet-n}(\At(O(2)^nND[n])).
\end{equation*}

\section{Cohomology of the elliptic tangent bundle}\label{sec:4}
Using the total radial residue map introduced in the previous section, we will compute the Lie algebroid cohomology of the elliptic tangent bundle.\\\\
Let $I_{\abs{D}}$ be an elliptic divisor of intersection number $n$, in this section we will write $\elli^n$ to keep track of this number and to simpify notion we write:
\begin{equation*}
\Omega^\bullet_{\text{res}}(\elli^n) := \Omega^\bullet(\At(Q));\det(L_{r}^*)),
\end{equation*}
and denote its cohomology by $H^{\bullet}_{\text{res}}(\elli^n)$. 

Recall that by Lemma \ref{lem:restrict} if $i \leq n$, then the restriction of $I_{\abs{D}}$ to $M\backslash D(i+1)$ defines an elliptic divisor of intersection number $i$, which we denote by $I_{\abs{D\backslash D(i+1)}}$. Therefore the inclusion $\iota_i : M\backslash D(i+1) \hookrightarrow M$ induces a cochain map:
\begin{equation*}
\iota^*_i : \Omega^\bullet (\elli^n) \rightarrow \Omega^{\bullet}(\mathcal{A}^{i}_{\abs{D \backslash D(i)}}).
\end{equation*}
Consequently, we can consider the radial residue map for the divisor $I_{\abs{D\backslash D(i+1)}}$ and the composition:
\begin{equation*}
\text{Res}_r^{i}\comp \iota^*_i : \Omega^{\bullet}(\elli^n) \rightarrow \Omega^{\bullet-i}_{\text{res}}(\mathcal{A}^{i}_{\abs{D \backslash D(i+1)}}),
\end{equation*}
for each $1 \leq i \leq n-1$.
\begin{theorem}\label{th:comology}
Let $I_{\abs{D}}$ be an elliptic divisor with intersection number $n$. Then
\begin{align}\label{eq:cohom}
H^k(\elli^n) &\rightarrow H^k(M \backslash D) \oplus \bigoplus_{i=1}^{n} H^{k-i}_{\text{\emph{res}}}(\mathcal{A}^{i}_{\abs{D \backslash D(i+1)}})\\
[\alpha] &\mapsto (\iota_0^*[\alpha],\Res_r^1\comp \iota_1^*[\alpha],\ldots,\Res_r^{n-1}\comp \iota_{n-1}^*[\alpha],\Res_r^n[\alpha])\nonumber
\end{align}
is an isomorphism.
\end{theorem}

\begin{proof}
The argument uses the observation from \cite{Del71}, that it suffices to proof that the above map induces an isomorphism in sheaf cohomology. Therefore, let $x \in D[l]$ and let $U$ be a neighbourhood of $x$ as in Lemma \ref{lem:ellinormform}, that is $U \cap D[l+1] = \emptyset$, $U = \rr^{2l}\times \rr^m$ and $I_{\abs{D}}$ is the standard elliptic divisor of intersection number $l$; $I_{\abs{D_1}}\cdot \ldots \cdot I_{\abs{D_l}}$.

Below we will implicitly push-forward sheaves using the inclusion maps $\iota_j$ whenever required.

%
%
% By Lemma \ref{lem:locMB} we may assume that we are in the setting of Example \ref{ex:mainexampleelliptic}, so that the divisor is a global product for which $C(1)$ is co-orientable. Therefore $\det(L_r^*)$ is trivial and by Proposition \ref{prop:struredsecond} $NC(i)$ has structure group $T^i$. 
% 
% By Lemma \ref{lem:Atiyahliealg} we conclude that $\Omega^\bullet_{\text{res}}(C(i),\mathcal{A}^{i}_{\abs{D \backslash D(i+1)}}) \simeq \Omega^{\bullet-i}(T^iNC(i))$.
% 
The left hand side of equation \eqref{eq:cohom} can be easily described as the following free algebra:

\begin{equation*}
	H^{\bullet}(U,\elli^n) = \inp{1,d\log r_1,\ldots,d\log r_l,d\theta_1,\ldots,d\theta_l}.
\end{equation*}
Now let $i \leq l$. By Lemma \ref{lem:gncresmap} $\Omega_{\res}^{\bullet}(\mathcal{A}^i_{\abs{D\backslash D(i+1)}}) = \Omega^{\bullet}(\At(O(2)^iND[i]))$. Because $D[i]$ is orientable, Lemma \ref{lem:Atiyahliealg} furthermore implies $\Omega_{\res}^{\bullet}(\mathcal{A}^i_{\abs{D\backslash D(i+1)}})= \Omega^{\bullet}(\At(T^iND[i]))$. And using invariant invariant cohomology $H^{\bullet}(\At(T^iND[i])) = H^{\bullet}_{\text{dR}}(T^iND[i])$. We will thus view $\Res_r^i\comp\iota_i^*$ as a map into the latter cohomology.

First, let us consider the map $\iota_0^*$. By an elementary argument $U\backslash D$ is homotopic to $T^l$, and $\iota_0^*$ sends the forms $\alpha$, with $\alpha \in \inp{1,d\theta_1,\ldots,d\theta_l}$ precisely to the generators of $H^{\bullet}(T^l)$, proving that $\iota_0^*$ is surjective. For $i> l$, $\iota^*_i : H^{\bullet}(U,\elli^n) \rightarrow H^{\bullet}_{\text{res}}(U,\mathcal{A}^i_{\abs{D\backslash D(i+1)}})$ is the zero map as $\mathcal{A}^i_{\abs{D\backslash D(i+1)}} = TU$ on $U$ and $U$ is contractible.

 We will prove that all maps $\iota_i^*\comp \Res_r^i$ with $1\leq i \leq l$ are surjective and have disjoint kernels.

First we need to establish the homotopy type of $D[i]$; remark that
\begin{equation*}
D[i] = \overset{\cdot}{\bigcup_{(j_1,\ldots,j_i)}} (D_{j_1}\cap \cdots \cap D_{j_i})\backslash D(i+1),
\end{equation*}
where the sum runs over all multi-indices of length $i$. By an elementary argument, $(D_{j_1}\cap \cdots \cap D_{j_i})\backslash D(i+1)$ is homotopic to $T^{l-i}$. Moreover, $ND[i]$ decomposes as a direct sum, and thus so does the associated $T^i$-bundle:
\begin{equation*}
T^iND[i] = \bigoplus_{(j_1,\ldots,j_i)} S^1ND_{j_1}|_{D[i]} \oplus \cdots \oplus S^1ND_{j_i}|_{D[i]}
\end{equation*}
Because each $D_{j_k}$ is contractible, the above is a sum of trivial circle bundles, and thus $T^iND[i] = T^i\times D[i]$. Combining this with the description of the homotopy type of $D[i]$ we conclude that:
\begin{equation}\label{eq:whatever}
H^{\bullet}_{\text{res}}(U,\mathcal{A}^i_{\abs{D\backslash D(i+1)}}) \simeq \bigoplus_{(j_1,\ldots,j_i)} H^{\bullet}(T^l).
\end{equation}
Given $\alpha \in \Omega^k(\mathcal{A}^i_{\abs{D\backslash D(i+1)}})$, then $\Res^i_r(\alpha)$ is only non-zero if $\alpha$ contains a term of the form $d\log r_{j_1} \wedge \cdots \wedge d\log r_{j_i} \wedge \beta$, with $\beta \in  \langle 1,d\theta_1,\ldots d\theta_l \rangle$. Moreover, $\Res_r^i$ takes forms of this form precisely to the generators of the $H^{\bullet}(T^n)$ in \eqref{eq:whatever} corresponding to the multi-index $(j_1,\ldots,j_i)$. This shows that $\iota_i^*\comp\Res_r^{i}$ is surjective, and that the kernels of these maps for different values of $i$ are all disjoint. We conclude that $(\iota_0^*,\iota_1^*\comp \Res_r^1,\ldots,\iota_{n-1}^* \comp \Res_r^{n-1},\Res_r^n)$ is an isomorphism.
\end{proof}

When the elliptic divisor is a co-orientable normal crossing divisor the cohomology becomes easier to describe:

\begin{corollary}\label{cor:main}
If $I_{\abs{D}}$ is a global normal crossing elliptic divisor for which $D[1]$ is co-orientable. Then
\begin{equation*}
H^k(\elli) \simeq H^k(M \backslash D) \oplus \bigoplus_{i=1}^n H^{k-i}(T^iN(D[i])).
\end{equation*}
\end{corollary}
\begin{proof}
By Corollary \ref{prop:strucredfirst} $ND[i]$ admits a $T^i$-structure group reduction, and using Lemma \ref{lem:Atiyahliealg} we see that $\Omega_{\Res}^{\bullet}(\mathcal{A}^i_{\abs{D\backslash D(i+1)}}) = \Omega^{\bullet}(\At(T^iND[i]))$.
Therefore by Theorem \ref{th:comology} we obtain that the cohomology is isomorphic to
\begin{align*}
H^k(\elli) \simeq H^k(M \backslash D) \oplus \bigoplus_{i=1}^n H^{k-i}(\At(T^iN(D[i])))
\end{align*}
Using invariant cohomology we have $H^{\bullet}(\At(T^iND[i])) = H^{\bullet}_{\text{dR}}(T^iND[i])$, and therefore we arrive at the required result.
\end{proof}
In applications we are often interested in the case when the manifold is four-dimensional, and hence $D[2]$ consists of a collection of points. In this case the cohomology also simplifies for general divisors:
\begin{corollary}\label{cor:final}
Let $I_{\abs{D}}$ be an elliptic divisor on $M^4$ for which $D[1]$ is co-orientable. Then
\begin{equation*}
H^k(\elli) \simeq H^k(M\backslash D) \oplus H^{k-1}(S^1ND[1]) \oplus  H^{k-2}(T^2)^{(\# D[2])}.
\end{equation*}
\end{corollary}
\section{Examples and outlook}\label{sec:5}
In this section we give some explicit examples of elliptic divisors and compute their associated cohomology. One class of examples arises through complex divisors:
\begin{definition}
A \textbf{smooth complex log divisor} $(L,\sigma)$ consists of a complex line bundle $L$ together with a transversely vanishing section $\sigma$.
\end{definition}

\begin{definition}
A \textbf{(self-crossing) complex log divisor} $(L,\sigma)$ is a complex divisor on $M$ with the property that for every point $p \in M$ there exists a neighbourhood $U$ of $p$ such that 
\begin{align*}
I_\sigma(U) = I_1\cdots \ldots \cdots I_j,
\end{align*}
where the $I_1,\ldots,I_j$ correspond to complex log divisors with transversely intersecting vanishing loci.
\end{definition}

\begin{remark}
Let $(L,\sigma)$ be a complex log divisor on $M$. Then then $(L\otimes \bar{L},\sigma \otimes \bar{\sigma})$ is the complexification of a real line bundle $(R,q)$. For every point $p \in D$ there exists local coordinates $z_1,\ldots,z_n$ such that $\sigma$ corresponds to the function $z_1 \cdot \ldots \cdot z_n$. Consequently, $q$ corresponds to the function $\abs{z_1}^2\cdot \ldots \cdot \abs{z_n}^2$ and thus $(R,q)$ defines an elliptic divisor. We call this the \textbf{associated elliptic divisor}.
\end{remark}

Just as for the elliptic tangent bundle, we can define a Lie algebroid associated to a complex divisor:
\begin{definition}
Let $(L,\sigma)$ be a complex log divisor. The complex vector fields presering $I_{\sigma}$ define a complex Lie algebroid $\mathcal{A}_D \rightarrow T_{\cc}M$, called the \textbf{complex log tangent bundle}.
\end{definition}
In \cite{CKW20} we use this Lie algebroid to study certain generalized complex structures. There, we also compute the Lie algebroid cohomology:
\begin{theorem}
Let $(L,\sigma)$ be a complex log divisor on a manifold $M$. Then the inclusion $\iota : M\backslash D \hookrightarrow M$ induces an isomorphism:
\begin{align*}
H^{\bullet}(\mathcal{A}_D) \simeq H^{\bullet}(M\backslash D;\cc).
\end{align*}
\end{theorem}

The following example plays an important role in constructing examples of stable generalized complex structures in \cite{CKW20}:
\begin{example}
Let $M = \mathbb{C}P^2$ and let $D$ be a normal crossings divisor consisting of three lines intersecting transversely, for instance $D = \set{z_0z_1z_2 = 0}$. The associated holomorphic line bundle with section $(L,\sigma)$ is a complex log divisor, and let $(R,q)$ be the associated elliptic divisor. We will now compute the cohomology of the associated elliptic tangent bundle:
\begin{itemize}
\item The complement $M\backslash D = (\mathbb{C}^*)^2$ and thus homotopic to $T^2$.
\item The stratum $D[1]$ consists of three disjoint cylinders, and therefore every complex line bundle on it is trivial. Consequently $S^1ND[1]$ is homotopic to three disjoint tori.
\item The stratum $D[2]$ consists of three points.
\end{itemize}
Corollary \ref{cor:final} now gives us the following description of the elliptic cohomology:
\begin{center}
\begin{tabular}{|l|l|l|l|l|l|l|}
\hline
$k$ & 0 & 1 & 2 & 3 & 4 & otherwise \\
\hline
$H^k(\elli)$ & $\rr$ & $\rr^5$ & $\rr^{10}$ & $\rr^9$ & $\rr^3$ & 0\\
\hline 
\end{tabular}
\end{center}
The comohology of the complex log tangent bundle is much simpler and given by:
\begin{center}
\begin{tabular}{|l|l|l|l|l|}
\hline
$k$ & 0 & 1 & 2 & otherwise \\
\hline
$H^k(\mathcal{A}_D)$ & $\cc$ & $\cc^2$ & $\cc$ & 0\\
\hline
\end{tabular}
\end{center}
\end{example}
The following example plays an important role in the construction of codimension-one symplectic foliations in work in progress of Cavalcanti and Crainic:
\begin{example}
Let $M = \mathbb{C}P^2$ and let $D$ be the smooth complex divisor given by the zero-set of a smooth cubic, for instance $D = \set{z_0^3+z_1^3+z_2^3=0}$. By the genus-degree formula we know that $D$ is isomorphic to a torus. The cohomology of the complement can be computed using the long exact sequence associated to the inclusion $D \hookrightarrow M$:
\begin{align*}
0 \rightarrow \rr \overset{\text{id}}{\rightarrow} \rr \rightarrow H^0(\cc P^2,D) \rightarrow 0 \rightarrow \rr^2 \rightarrow H^1(\cc P^2,D) \rightarrow \rr \overset{f}{\rightarrow} \rr \rightarrow H^2(\cc P^2,D) \rightarrow 0.
\end{align*}
Because $D$ is a complex submanifold of a Kahler manifold it is also a symplectic submanifold with respect to $\omega_{FS}$. Therefore the map $H^2(\mathbb{C}P^2) \rightarrow H^2(D)$ is an isomorphism, and the map $f$ becomes the identity, and we can easily read of the cohomology of $M\backslash D$ form the above sequence. To compute the cohomology of $P = S^1ND$ we employ the Thom-Gysin sequence, which takes the following form:
\begin{align*}
0 \rightarrow \rr \rightarrow H^0(P) \rightarrow 0 \rightarrow \rr^2 \rightarrow H^1(P) \rightarrow \rr \overset{\cup ~ e}{\rightarrow} \rr \rightarrow H^2(P) \rightarrow \rr^2 \rightarrow 0.
\end{align*}
We have $ND = \mathcal{O}(3)|_D$, because $\mathcal{O}(3)$ is non-trivial and $H^2(\mathbb{C}P^2) \rightarrow H^2(D)$ is an isomorphism, we have that the Euler-class of $P$ is non-trivial. Therefore the map $\cup ~e$ is an isomorphism, and we can easily read of the cohomology of $P$ from the sequence.
Corollary \ref{cor:final} now gives us the following description of the elliptic cohomology:
\begin{center}
\begin{tabular}{|l|l|l|l|l|l|l|}
\hline
$k$ & 0 & 1 & 2 & 3 & otherwise \\
\hline
$H^k(\elli)$ & $\rr$ & $\rr^3$ & $\rr^2$ & $\rr^2$ & 0\\
\hline 
\end{tabular}
\end{center}

Again, this is quite different from the complex log tangent bundle:
\begin{center}
\begin{tabular}{|l|l|l|l|l|}
\hline
$k$ & 0 & 1 & otherwise \\
\hline
$H^k(\mathcal{A}_D)$ & $\cc$ & $\cc^2$ & 0\\
\hline
\end{tabular}
\end{center}
\end{example}

\begin{example}
Let $f : M^4 \rightarrow \Sigma^2$ be a Lefschetz fibration with singular value $p$ and corresponding singular fibre $F= f^{-1}(p)$. Let $(E,s)$ be the complex divisor on $\Sigma$ corresponding to the divisor $\set{p} \subset \Sigma$, and define a self-crossing complex log divisor on $M$ by $(f^*E,f^*s)$. This divisor has as vanishing locus precisely the singular fibre $F$, and $D[2]$ consist precisely of the critical points on $F$.

Let $n$ be the number of critical points of $F$. Then $D[1]$ consist of a disjoint union of $n$ cylinders, and consequently $S^1ND[1]$ is homotopic to $n$ tori. 

Corollary \ref{cor:final} now gives us the following description of the elliptic cohomology:
\begin{align*}
H^{k}(\elli) = H^k(M\backslash F) \oplus \begin{tabular}{|l|l|l|l|l|l|}
\hline
$k$ & 1 & 2 & 3 & 4 & otherwise \\
\hline
 & $\rr^n$ & $\rr^{3n}$ & $\rr^{3n}$ & $\rr^n$ & 0\\
\hline
\end{tabular}
\end{align*}

\end{example}
\subsection{Outlook: Deformation theory}
We end by given an outlook on studying the deformation theory of self-crossing stable generalized complex structures and self-crossing elliptic symplectic structures.

Already the deformation theory of all smooth elliptic symplectic structures is more difficult than that of smooth stable generalized complex structures. The reason is that smooth complex log divisors are very rigid: if one deforms a complex function which vanishes transversely a bit, than one again obtains a function which vanishes transversely. This is used in \cite{CG17} to show that any two smooth complex log divisors which are deformations of each other are in fact isomorphic. Therefore, deformations of smooth stable generalized complex structures can be reduced to deformations on a fixed Lie algebroid, simplifying the discussion significantly. In the end, this ensures that deformations of smooth stable generalized complex structures are completely controlled by the cohomology of the Lie algebroid.

Smooth elliptic divisors are much less rigid than smooth complex log divisors. If $I_{\abs{D}}$ is a smooth elliptic ideal with non-negative generator $f$ then for all $\epsilon > 0$, the function $f + \varepsilon$ will be nowhere vanishing. Therefore small deformations of elliptic divisors will not necessarily be isomorphic to the original divisor, which significantly complicates the deformation theory of general elliptic symplectic structures.

Allowing for self-crossings allows for even more flexibility, so that in this case also complex log divisors are no longer rigid. The reason is that the self-crossings can be smoothend. Take for instance the complex divisor $I_D = \inp{z_1 z_2}$ on $\cc^2$, which has as vanishing locus the coordinate hyperplanes. The divisor $I_{D_t} = \inp{z_1 z_2 + t}$ is a smooth complex log divisor for all $t \in [0,1]$ and provides a deformation between $I_D$ and a complex log divisor with embedded vanishing locus. This shows that for self-crossing complex log divisors, deformation equivalent divisors need not be isomorphic.

One might restrict the discussion to deformations for which the divisors are isomorphic. We have good reasons to believe that in that case the results of \cite{CG17} can be generalized and the deformations will be controlled by the cohomology. However, deformations as the one above are very important to the theory. The particular deformation $I_{D_t}$ is used in \cite{CKW20} to show that in four dimensions any self-crossing stable generalized complex structure can be deformed into one with embedded degeneracy locus.

Therefore it is most natural to study general deformations, which would thus combine the deformation theory of the Lie algebroid with that of the elliptic symplectic structures. To study this one will have to combine the cohomology governing the deformations of the Lie algebroid, which is the deformation cohomology of Crainic-Moerdijk (\cite{CM08}) together with the Lie algebroid cohomology of the elliptic tangent bundle.

\appendix
\section{Atiyah algebroids}\label{sec:Atiyah}
This section recalls some results on Atiyah algebroids needed in this paper. These results are all classical and can be found in the literature, see for instance \cite{CF03}.
 
\begin{definition}
Let $P$ be a principal $G$-bundle over $M$. The \textbf{Atiyah algebroid} of $P$, is defined as $\At(P) = TP/G$. The anchor is induced by the differential of the projection $P\rightarrow M$. The bracket is obtained by viewing sections of $\At(P)$ as $G$-invariant vector fields on $P$.
\end{definition}

Given a vector bundle $E$, we can consider its frame bundle $\text{Fr}(E)$ and its associated Atiyah algebroid, which we will denote by $\At(E)$. We can describe its sections as fibre-wise linear vector fields:
\begin{lemma}
Let $E \rightarrow M$ be a vector bundle. Then $\Gamma(\At(E)) \simeq \mathfrak{X}^1(E)_{\lin} = \set{X \in \mathfrak{X}^1(E) : [X,\mathcal{E}] = 0}$.
\end{lemma}

We are mainly interested in the case of metrics, for which we have the following descriptions:
\begin{lemma}\label{lem:metricautos}
Given a metric $g$ on a vector bundle $E$ let $Q$ denote the corresponding $O(n)$ structure group reduction of $\Fr(E)$. We have
\begin{align*}
\Gamma(\At (Q)) \simeq  \set{\tilde{X} \in \mathfrak{X}^1(E)_{\lin} : \mathcal{L}_{\tilde{X}}g = 0}.
\end{align*}
\end{lemma}
\begin{lemma}\label{lem:confmetautos}
Given a conformal class of metrics $[g]$ let $Q$ denote the corresponding $O(n)\times \rr_+$ structure group reduction of $\Fr(E)$. We have
\begin{align*}
\Gamma(\At(Q)) \simeq  \set{\tilde{X} \in \mathfrak{X}^1(E)_{\lin} : \mathcal{L}_{\tilde{X}}g = \lambda g, ~ \text{for some }\lambda \in C^{\infty}(M)}.
\end{align*}
\end{lemma}
The Atiyah algebroid cannot detect discrete parts of the structure group:
\begin{lemma}\label{lem:Atiyahliealg}
If $P$ is a principal $G$-bundle and $Q$ is a structure group reduction to $H$, and $\mathfrak{g} \simeq \mathfrak{h}$ then $\At(P) \simeq \At(Q)$.
\end{lemma}
\bibliographystyle{hyperamsplain-nodash}
\bibliography{references} 

\providecommand{\bysame}{\leavevmode\hbox to3em{\hrulefill}\thinspace}
\providecommand{\MR}{\relax\ifhmode\unskip\space\fi MR }
\providecommand{\MRhref}[2]{%
  \href{http://www.ams.org/mathscinet-getitem?mr=#1}{#2}
}
\providecommand{\href}[2]{#2}
\begin{thebibliography}{10}

\bibitem{CG17}
G.~R. Cavalcanti and M.~Gualtieri, \emph{Stable generalized complex
  structures}, \href{http://dx.doi.org/10.1112/plms.12093}{Proc. Lond. Math.
  Soc. \textbf{116} (2018)}, no.~5, 1075--1111.

\bibitem{CKW20}
G.~R. Cavalcanti, R.~L. Klaasse, and A.~Witte, \emph{Self-crossing stable
  generalized complex structures}, \href{http://arxiv.org/abs/2004.07559}{{\tt
  arXiv:2004.07559 [math.DG]}}.

\bibitem{CF03}
M.~Crainic and R.~L. Fernandes, \emph{Integrability of {L}ie brackets},
  \href{http://dx.doi.org/10.4007/annals.2003.157.575}{Ann. of Math. (2)
  \textbf{157} (2003)}, no.~2, 575--620.

\bibitem{Del71}
P.~Deligne, \emph{Th\'{e}orie de {H}odge. {II}}, Inst. Hautes \'{E}tudes Sci.
  Publ. Math. (1971), no.~40, 5--57.

\bibitem{G13}
R.~Goto, \emph{Unobstructed Deformations of Generalized Complex Structures
  Induced by $C^{\infty}$ Logarithmic Symplectic Structures and Logarithmic
  Poisson Structures}, Geometry and Topology of Manifolds (Tokyo) (A.~Futaki,
  R.~Miyaoka, Z.~Tang, and W.~Zhang, eds.), Springer Japan, 2016, pp.~159--183.

\bibitem{Gua07}
M.~Gualtieri, \emph{Generalized complex geometry},
  \href{http://dx.doi.org/10.4007/annals.2011.174.1.3}{Ann. of Math. (2)
  \textbf{174} (2011)}, no.~1, 75--123.

\bibitem{GMP14}
V.~Guillemin, E.~Miranda, and A.~R. Pires, \emph{Symplectic and Poisson
  geometry on b-manifolds},
  \href{http://dx.doi.org/https://doi.org/10.1016/j.aim.2014.07.032}{Advances
  in Mathematics \textbf{264} (2014)}, 864--896.

\bibitem{H03}
N.~Hitchin, \emph{Generalized {C}alabi-{Y}au manifolds},
  \href{http://dx.doi.org/10.1093/qjmath/54.3.281}{Q. J. Math. \textbf{54}
  (2003)}, no.~3, 281--308.

\bibitem{K18}
R.~L. Klaasse, \emph{Poisson structures of divisor-type},
  \href{http://arxiv.org/abs/1811.04226}{{\tt arXiv:1811.04226}}. 57 pages.

\bibitem{Lan21}
M.~Lanius, \emph{Symplectic, Poisson, and contact geometry on scattering
  manifolds}, \href{http://dx.doi.org/10.2140/pjm.2021.310.213}{Pacific Journal
  of Mathematics \textbf{310} (2021)}, no.~1, 213–256.

\bibitem{MS20}
E.~Miranda and G.~Scott, \emph{The geometry of $E$-manifolds},
  \href{http://dx.doi.org/10.4171/rmi/1232}{Revista Matemática Iberoamericana
  \textbf{37} (2020)}, no.~3, 1207–1224.

\bibitem{CM08}
I.~Moerdijk and M.~Crainic, \emph{Deformation of Lie brackets: cohomological
  aspects}, \href{http://dx.doi.org/10.4171/JEMS/139}{Journal of the European
  Mathematical Society \textbf{10} (2008)}.

\bibitem{Pym17}
B.~Pym, \emph{Constructions and classifications of projective Poisson
  varieties}, \href{http://dx.doi.org/10.1007/s11005-017-0984-5}{Letters in
  Mathematical Physics \textbf{108} (2017)}, no.~3, 573–632.

\end{thebibliography}
\end{document}